\documentclass[a4paper,leqno,12pt]{article}

\usepackage[T1]{fontenc}
\usepackage[utf8]{inputenc}

\usepackage{geometry}
\usepackage[draft]{todonotes}

\usepackage[shortlabels]{enumitem}
\usepackage{amssymb, amsmath, amsthm}
\usepackage{faktor}
\usepackage{nicefrac}

\usepackage{xcolor}
\usepackage[colorlinks]{hyperref}
\definecolor{links}{RGB}{76,153,0}
\hypersetup{
    colorlinks=true,
    linkcolor=red,
    filecolor=red,      
    urlcolor=blue,
    citecolor=links,
}
\usepackage[reftex]{theoremref}

\usepackage[all]{xy}
\usepackage{tikz-cd}

\usepackage[style=alphabetic]{biblatex}
\DeclareFieldFormat{eprint}{\href{https://arxiv.org/abs/#1}{\texttt{#1}}}

\renewbibmacro*{date}{%
\printdate
\iffieldundef{origyear}{%
}{%
  \setunit*{\addspace}%
  \printtext[parens]{\printorigdate}%
}%
}

\newcommand{\R}{\mathbb{R}}
\newcommand{\C}{\mathbb{C}}
\newcommand{\Z}{\mathbb{Z}}
\newcommand{\N}{\mathbb{N}}

\newcommand{\res}{\text{res}}

\newcommand{\supp}{\text{supp}}

\newcommand{\iso}{\text{Iso}}

\newcommand{\G}{\mathcal{G}}
\newcommand{\F}{\mathcal{F}}

\newcommand{\U}{\mathcal{U}}

\newcommand{\E}{\widehat{E}}

\newcommand{\cG}{\ensuremath{\mathcal{G}}}
\newcommand{\cH}{\ensuremath{\mathcal{H}}}
\newcommand{\cK}{\ensuremath{\mathcal{K}}}
\newcommand{\rE}{\ensuremath{\mathrm{E}}}
\newcommand{\CO}{\ensuremath{\mathrm{CO}}}

\newcommand{\Tzero}{\ensuremath{{(\mathrm{T}_0)}}}
\newcommand{\Tone}{\ensuremath{{(\mathrm{T}_1)}}}

\newcommand{\Gnought}{\ensuremath{\G}^{(0)}}

\newcommand{\thra}{\ensuremath{\twoheadrightarrow}}
\newcommand{\ev}{\ensuremath{\mathrm{ev}}}

\newcommand{\Sym}{\ensuremath{\mathrm{Sym}}}
\newcommand{\Cstar}{\ensuremath{\text{C}^*}}

\newcommand{\im}{\ensuremath{\mathop{\mathrm{im}}}}
\renewcommand{\supp}{\ensuremath{\mathop{\mathrm{supp}}}}

\theoremstyle{plain}
\newtheorem{theorem}{Theorem}[section]
\newtheorem{lemma}[theorem]{Lemma}

\newtheorem{proposition}[theorem]{Proposition}

\newtheorem{introtheorem}{Theorem}

\theoremstyle{definition}
\newtheorem{definition}[theorem]{Definition}

\theoremstyle{remark}
\newtheorem{remark}[theorem]{Remark}

\newtheorem{example}[theorem]{Example}

\newcommand{\authors}{Gabriel Favre and Sven Raum}
\renewcommand{\title}{An algebraic characterization of ample type I groupoids }

\begin{document}

\thispagestyle{empty}

\noindent
\begin{minipage}{\linewidth}
  \textbf{\Large \title}%
  \begin{center}
    by \authors    
  \end{center}
\end{minipage}

\renewcommand{\thefootnote}{}
\footnotetext{
  \textit{MSC classification: 20M18; 22A22, 20M30, 06E15, 46L89}
}
\footnotetext{
  \textit{Keywords: inverse semigroup, ample groupoid, noncommutative Stone duality, CCR, type I}
}

\vspace{2em}
\noindent
\begin{minipage}{\linewidth}
  \textbf{Abstract}.
  We give algebraic characterizations of the type I and CCR properties for locally compact second countable, ample Hausdorff groupoids in terms of subquotients of its Boolean inverse semigroup of compact open bisections. It yields in turn algebraic characterizations of both properties for inverse semigroups in terms of subquotients of their booleanization.
\end{minipage}

\section{Introduction}

Inverse semigroups provide an algebraic framework to study partial dynamical systems and groupoids.  This relation is described by noncommutative Stone duality, extending the classical duality between totally disconnected Hausdorff spaces and generalized Boolean algebras to a duality between ample Hausdorff groupoids and so called Boolean inverse semigroups.  This corner stone of the modern theory of inverse semigroups was introduced by Lawson and Lenz \cite{lawson2010, lawson2012, lawsonlenz13} following ideas of Exel \cite{exel2009} and Lenz \cite{lenz2008}.  In a nutshell, noncommutative Stone duality associates to an ample Hausdorff groupoid the Boolean inverse semigroup of its compact open bisections.  Various filter constructions can be employed to describe the inverse operation \cite{lawson2012, lawsonmargolissteinberg2013, armstrongclarkhuefjoneslin2020}.

Since its discovery of noncommutative Stone duality, an important aspect of this theory was to establish a dictionary between properties of ample groupoids and Boolean inverse semigroups.  A summary of the results obtained so far, can be found in Lawson's survey article \cite{lawson2019-survery}.  From an operator algebraic and representation theoretic point of view, one fundamental side of groupoids and semigroups is the property of being CCR or type I.  These notions arise upon considering groupoid and semigroup \Cstar-algebras.  So far, neither of these properties has been addressed by the dictionary of noncommutative Stone duality.  Our first two main results fill this gap and establish algebraic characterizations of CCR and type I Boolean inverse semigroups, matching Clark's characterization of the respective property for groupoids \cite{clark07}.

Our characterisation roughly takes the form of forbidden subquotients, in analogy with the theme of forbidden minors in graph theory and other fields of combinatorics.  The Boolean inverse semigroup $B_\Tone$ featuring in the next statement is introduced in \th\ref{ex:btone}.  It is the algebraically simplest possible Boolean inverse semigroup which is not CCR.  We refer the reader to Sections \ref{sec:inverse-semigroups} and \ref{sec:nc-stone-duality} for further information about Boolean inverse semigroups, their corners and group quotients.

\begin{introtheorem}
  \th\label{introthm:ccr-groupoid}
  Let $\G$ be a second countable, ample Hausdorff groupoid. Then $\G$ is CCR if and only if the following two conditions are satisfied.
  \begin{itemize}
  \item No corner of $\Gamma(\G)$ has a non virtually abelian group quotient, and
  \item $\Gamma(\G)$ does not have $B_\Tone$ as a subquotient.
  \end{itemize}
\end{introtheorem}

\begin{introtheorem}
  \th\label{introthm:type-I-groupoid}
  Let $\G$ be a second countable, ample Hausdorff groupoid. Then $\G$ is type ${\rm I}$ if and only if the following two conditions are satisfied.
  \begin{itemize}
  \item No corner of $\Gamma(\G)$ has a non virtually abelian group quotient, and
  \item $\Gamma(\G)$ does not have an infinite, monoidal and $0$-simplifying subquotient.
  \end{itemize}
\end{introtheorem}

Historically the notation of CCR and type I \Cstar-algebras was motivated by problems in representation theory.  Roughly speaking, groups enjoying one of these properties have a well behaved unitary dual.  Originally studied in the context of Lie groups and algebraic groups \cite{harish-chandra53, dixmier57, bernstein74-type-I,bekkaechterhoff2020}, other classes of non-discrete groups were considered more recently \cite{ciobotaru15,houdayerraum16-non-amenable}.  The question which discrete groups are CCR and type I was answered conclusively by Thoma \cite{thoma68}, characterizing them as the virtually abelian groups.  A more direct proof of Thoma's result was obtained by Smith \cite{smith72} and his original proof was the basis for a Plancherel formula for general discrete groups recently obtained by Bekka \cite{bekka2020-plancherel}.  The fundamental nature of CCR and type I \Cstar-algebras also led to study other group like objects such as the aforementioned characterisation of groupoids with these property by Clark \cite{clark07}.  Compared with Thoma's characterisation of discrete type I groups, our Theorems \ref{introthm:ccr-groupoid} and \ref{introthm:type-I-groupoid} can be considered an analogue for Boolean inverse semigroups.  Naturally, the question arises whether a similar characterisation can be obtained for inverse semigroups.  The bridge between these structures is provided by the booleanization of an inverse semigroup \cite{lawsonlenz13, lawson2019}, which we denote by $B(S)$.  In view of a direct algebraic construction of the booleanization we expose in Section \ref{sec:inverse-semigroups}, the next two results give an intrinsic characterisation of CCR and type I inverse semigroups.
\begin{introtheorem}
  \th\label{introthm:ccr-semigroup}  
Let $S$ be a discrete inverse semigroup. Then $S$ is CCR if and only if the following two conditions are satisfied
\begin{enumerate}
    \item $S$ does not have any non virtually abelian group subquotient, and
    \item $B(S)$ does not have $B_{(T_1)}$ as a subquotient.
\end{enumerate}
\end{introtheorem}

\begin{introtheorem}
  \th\label{introthm:type-I-semigroup}  
  Let $S$ be an inverse semigroup. $S$ is type I if and only if the following two conditions are satisfied
\begin{enumerate}
    \item $S$ does not have any non virtually abelian group subquotient, and
    \item $B(S)$ does not have an infinite, monoidal, $0$-simplifying subquotient.
\end{enumerate}
\end{introtheorem}
The booleanization of an inverse semigroup was first introduced in \cite{lawson2019}. Our description is more natural from an operator algebraic point of view and equivalent to the original definition of Lawson.

This paper contains five sections.  After the introduction, we expose necessary preliminaries on (Boolean) inverse semigroups, ample groupoids and noncommutative Stone duality.  In Section \ref{sec:separation-properties}, we study separation properties of the orbit space of a groupoid and obtain algebraic characterisations of groupoids with $\Tone$ and $\Tzero$ orbit spaces. In Section~\ref{sec:isotropy-groups}, we relate isotropy groups of a groupoid to subquotients of the associated Boolean inverse semigroup.  The proofs of our main results are collected in Section \ref{sec:main-proofs}.

\subsection*{Acknowledgements}

S.R.\ was supported by the Swedish Research Council through grant number 2018-04243 and by the European Research Council (ERC) under the European Union's Horizon 2020 research and innovation programme (grant agreement no. 677120-INDEX).  He would like to thank Piotr Nowak and Adam Skalski for their hospitality during the authors stay at IMPAN.  G.F.\ would like to thank Piotr Nowak for the hospitality and financial support during his visit to IMPAN in autumn 2020.

\section{Preliminaries}
\label{sec:preliminaries}

In this section, we recall all notations relevant to our work and introduce some elementary constructions that will be important throughout the text.  For inverse semigroups the standard reference is \cite{lawson1998}.  The survey article \cite{lawson2019-survery} describes recent advances and the state of the art in inverse semigroup theory.  Further, \cite{lawson2019} and \cite[Sections 3 and 4]{lawson2019-survery} provide an introduction to Boolean inverse semigroups.  It has to be pointed out that the definition of Boolean inverse semigroups used in the literature has changed over the years, so that care is due when consulting older material.  The standard reference for groupoids and their C$^*$-algebras is \cite{renault80}.  For groupoids attached to inverse semigroups and Boolean inverse semigroups, we refer to \cite{li2017-mfo, lawson2012}.

\subsection{Inverse semigroups and Boolean inverse semigroups}
\label{sec:inverse-semigroups}

In this section we recall the notions of inverse semigroups, Boolean inverse semigroups and the link between them provided by the universal enveloping Boolean inverse semigroup of an inverse semigroup, termed booleanization.

An \emph{inverse semigroup} $S$ is a semigroup in which for every element $s \in S$ there is a unique element $s^* \in S$ satisfying $ss^*s = s$ and $s^*ss^* = s^*$.  The set of idempotents $E(S) \subseteq S$ forms a commutative meet-semilattice, when endowed with the partial order $e \leq f$ if $e f = e$.  We denote by $\supp s = s^*s$ and $\im s = ss^*$ the support and the image of an element $s \in S$, which are idempotents.  The partial order on $E(S)$ extends to $S$ by declaring $s \leq t$ if $\supp s \leq \supp t$ and $t \supp s = s$.  Given an inverse semigroup $S$, we denote by $S_0$ the inverse semigroup with zero obtained by formally adjoining an absorbing idempotent $0$.  In particular, we will make use of groups with zero.  A \textit{character} on $E(S)$ is a non-zero semilattice homomorphism to $\{ 0, 1 \}$.   We will denote by $\widehat{E(S)}$ the space of characters on $E(S)$ equipped with the topology of pointwise convergence.

Let $S$ be an inverse semigroup with zero.  Two elements $s,t \in S$ are called \emph{orthogonal}, denoted $s \perp t$, if $s t^* = 0 = s^*t$.  A \emph{Boolean inverse semigroup} is an inverse semigroup with zero, whose semilattice of idempotents is a generalized Boolean algebra such that finite families of pairwise orthogonal elements have joins.  Recall that a generalized Boolean algebra can be conveniently described as a Boolean rng. Given two Boolean inverse semigroups $B$ and $C$, and an inverse semigroup morphism $\phi : B \to C$, we say that $\phi$ is a morphism of Boolean inverse semigroups if it preserves the joins of orthogonal elements.

To every inverse semigroup $S$, one associates the \emph{enveloping Boolean inverse semigroup} or \emph{booleanization} $S \subseteq B(S)$, which satisfies the universal property that for every Boolean inverse semigroup $B$ and every semigroup homomorphism $S \to B$, there is a unique extension to $B(S)$ such that the following diagram commutes.
\begin{gather*}
  \xymatrix{
    S \ar[r] \ar[d] & B \\
    B(S) \ar[ur]
  }
\end{gather*}
The enveloping Boolean inverse semigroup is conveniently described as the left adjoint of the forgetful functor from Boolean inverse semigroups to inverse semigroups with zero \cite{lawson2019}.  We will use the following concrete description.  Consider a semigroup $S$.  The semigroup algebra $I(S) = F_2[E(S)]$ is a Boolean rng and whose characters (as a rng) are in one to one correspondence with characters of $E(S)$.  We consider the following set of formal sums.
\begin{gather*}
  C(S) = \{\sum_i s_i e_i  \mid s_i \in S, e_i \in I(S) \text{ and } (e_i)_i \text{ and } (s_ie_is_i^*)_i \text{ are pairwise orthogonal}\}.
\end{gather*}
We consider the equivalence relation given by the following condition.  We have $\sum_i s_i e_i \sim \sum_j t_j f_j$ if and only if $e_i f_j \neq 0$ implies that there is some $p \in E(S)$ such that $s_i p = t_j p$.  One readily checks that the quotient of $C(S)$ by this equivalence relation is a Boolean inverse semigroup.  Thanks to the existence of joins of orthogonal families, every map of semigroups with zero into a Boolean inverse semigroup $S_0 \to B$ extends uniquely to a map ${C(S)}/{\sim} \to B$.  By uniqueness of adjoint functors, this shows that ${C(S)}/{\sim} \cong B(S)$ is the booleanization of $S$.

\subsection{Ample groupoids}
\label{sec:ample-groupoids}

In this section we fix our notation for groupoids and recall some basic results.  We recommend \cite{renault80} and \cite{paterson1999} as resources on the topic.

Given a groupoid $\G$, we denote its set of units by $\Gnought$ and the range and source map by $r: \G \to \Gnought$ and $d: \G \to \Gnought$, respectively.  Throughout the text, $\G$ will be a \textit{topological groupoid} meaning it is equipped with a topology making the multiplication and inversion continuous.  A \textit{bisection} of a topological groupoid $\G$ is a subset $U \subseteq \G$ such that the restrictions $d|_U$ and $r|_U$ are homeomorphisms onto their images.  We call $\G$ \emph{{\'e}tale} if its topology has a basis consisting of bisections.  It is called \emph{ample} if its topology has a basis consisting of compact open bisections.

If $\G$ is a groupoid and $A \subseteq \Gnought$ is a set of units, we denote by $\G|_A = \{g \in \G \mid d(g), r(g) \in A\}$ the restriction of $\G$ to $A$.  It does not need to be {\'e}tale, even if $G$ is so.  If $\G$ is {\'e}tale and $A \subseteq \Gnought$ is open, then $G|_A$ is {\'e}tale too.

The isotropy at $x \in \Gnought$ is $\G|_x = \G|_{\{x\}}$.  We denote by $\iso(\G)$ the union over all isotropy groups, considered as subsets of $\G$.  Then $\G$ is \emph{effective} if the interior of $\iso(\G) \setminus \Gnought$ is empty. Given a unit $x \in \Gnought$, its orbit is denoted by $\G x \subseteq \Gnought$.  We call $\G$  \emph{minimal} if all its orbits are dense.

The set of orbits of a topological groupoid inherits a natural topology.  We will be interested in separation properties of this \emph{orbit space}.  First, we observe that the orbit space of a groupoid is a \Tone-space if and only if its orbits are closed.  An analogue characterisation of groupoids whose orbit space is a \Tzero-space, is the subject of the Ramsey-Effros-Mackey dichotomy, which we now recall.
\begin{proposition}
  \th\label{prop:tzero-characterisations}
  Let $\G$ be a second countable ample groupoid.  Then the following statements are equivalent.
  \begin{itemize}
  \item The orbit space of $\G$ is \Tzero.
  \item $\G$ has a self-accumulating orbit.
  \item Orbits of $\G$ are locally closed.
  \end{itemize}
\end{proposition}
\begin{proof}
  The equivalence between the first two items follows from \cite[Theorem 2.1, (2) and (4)]{ramsey1990}.  In order to prove the equivalence to the last item, we want to apply \cite[Theorem 2.1, (4) and (5)]{ramsey1990}. To this end, we have to show that the equivalence relation induced by $\G$ on $X = \Gnought$ is an $F_\sigma$-subset of $X \times X$.  The map $\G \to X \times X$ restricted to any compact open bisection of $\G$ has a closed image.  We conclude with the observation that there are only countably many compact open bisections, since $\G$ is second countable.
\end{proof}

\subsection{Noncommutative Stone duality}
\label{sec:nc-stone-duality}

Classical Stone duality establishes an equivalence of categories between locally compact totally disconnected Hausdorff topological spaces and generalized Boolean algebras.  If $X$ is such a space, the generalized Boolean algebra associated with it is $\CO(X)$, the algebra of compact open subsets of $X$.  Vice versa, given a generalized Boolean algebra $B$, its spectrum $\widehat{B}$ is a locally compact totally disconnected Hausdorff topological space.  Noncommutative Stone duality generalises this correspondence to an equivalence between ample Hausdorff groupoids and Boolean inverse semigroups.  We refer the reader to \cite{lawson2010,lawson2012}.

Given an ample Hausdorff groupoid $\G$, we denote by $\Gamma(\G)$ the set of compact open bisections of $\G$, which is a Boolean inverse semigroup.  Conversely, given a Boolean inverse semigroup $B$, the set of ultrafilter for the natural order on $B$ forms an ample Hausdorff groupoid $\G(B)$.  It follows from \cite[Duality Theorem]{lawson2012} that these two operations are dual to each other.  See also \cite[Theorem 4.4]{lawson2019-survery} and \cite[Theorem 4.2]{lawsonvdovina2019}.  We will not need to specify the morphisms of the categories involved in this duality.  Invoking \cite[Theorem 1.2]{lawsonvdovina2019}, the Paterson groupoid of an inverse semigroup $S$ can now be identified with $\G(B(S))$.

Noncommutative Stone duality establishes a dictionary between properties of Boolean inverse semigroups and ample Hausdorff groupoids.  We recall some of its aspects that will be needed in this piece.

\subsubsection*{Corners and subgroupoids}

Given an ample Hausdorff groupoid $\G$, idempotents in $\Gamma(\G)$ correspond to compact open subsets of $\G$.  Given such idempotent $p \in E(\Gamma(\G))$ corresponding to $U \subseteq \Gnought$, the \emph{corner} $p \Gamma(\G) p$ is naturally isomorphic with $\Gamma(\G|_U)$.  Further, there is a one-to-one correspondence between open subgroupoids of $\G$ and Boolean inverse subsemigroups $B \subseteq \Gamma(\G)$ assigning the groupoid $\bigcup B \subseteq \G$ to a Boolean inverse semigroup $B \subseteq \Gamma(\G)$.

\subsubsection*{Morphisms}

Noncommutative Stone duality does not cover all morphisms that one would naturally consider in the respective category.  This is why it is necessary and useful to note the following two statements. They ensure compatibility of noncommutative Stone duality with restriction maps.  The next proposition is a reformulation of \cite[Proposition 5.10]{lawsonvdovina2019}. See also \cite{lenz2008}.
\begin{proposition}
  \th\label{prop:restriction-induces-bis-map}
  Let $\G$ be an ample groupoid, $A \subseteq \Gnought$ a closed $\G$-invariant set. Then the restriction map $\CO(\Gnought) \to \CO(A)$ extends to a unique homomorphism $\Gamma(\G) \to \Gamma(\G|_A)$ with the universal property that for every other homomorphism $\pi: \Gamma(\G) \to B$ such that
  \begin{gather*}
    \xymatrix{
      \CO(\Gnought) \ar[r]^{\pi|_{\CO(\Gnought)}} \ar[d]_{\res_A} & E(B) \\
      \CO(A) \ar@{-->}[ur]
    }
  \end{gather*}
  commutes, there is a unique extension to a commutative diagram
  \begin{gather*}
    \xymatrix{
      \Gamma(\G) \ar[r]^{\pi} \ar[d]_{\res_A} & B \\
      \Gamma(\G|_A) \ar@{-->}[ur]
    }
  \end{gather*}
\end{proposition}

The following converse to \th\ref{prop:restriction-induces-bis-map} can be considered a reformulation of \cite[Lemma 5.6]{lawsonvdovina2019}.
\begin{lemma}
  \th\label{lem:restriction-implies-invariant}
  Let $\G$ be an ample groupoid and $A \subseteq \Gnought$ a closed subset, such that the restriction $\mathrm{res}_A: \mathrm{CO}(\Gnought) \to \mathrm{CO}(A)$ extends to a map of inverse semigroups $\Gamma(\G) \to B$, for some inverse semigroup $B$.  Then $A$ is $\G$-invariant.
\end{lemma}
\begin{proof}
  Denote by $\pi: \Gamma(\G) \to B$ a map as in the statement of the lemma.  Take $s \in \Gamma(\G)$ satisfying $\supp(s) \cap A = \emptyset$.  Then the calculation
  \begin{gather*}
    \mathrm{res}_A(ss^*) = \pi(ss^*) = \pi(s) \pi(s^*s) \pi(s^*) = 0
  \end{gather*}
  shows that $\im(s) \cap A = \emptyset$. So $A$ is indeed $\G$-invariant.
\end{proof}

\subsubsection*{Minimal groupoids and $0$-simplifying Boolean inverse semigroups}

We introduce the algebraic notion corresponding to minimality of groupoids following \cite{steinbergszakacs2020}.
\begin{definition}
  \th\label{def:zero-simplifying}
  Let $B$ be a Boolean inverse semigroup.  An ideal $I$ of $B$ is called \emph{additive}, if it is closed under joins of orthogonal elements.  Further, $B$ is called \emph{$0$-simplifying} if its only additive ideals are $\{0\}$ and $B$.
\end{definition}
The name of the notation $0$-simplifying stems from the fact that additive ideals are exactly of the form $\pi^{-1}(0)$ for homomorphisms of Boolean inverse semigroups $B \to C$.

\begin{proposition}[{\cite[Proposition 2.7]{steinbergszakacs2020}}]
  \th\label{prop:0simplifying}
  Let $\G$ be an ample Hausdorff groupoid. Then $\G$ is minimal if and only if $\Gamma(\G)$ is $0$-simplifying
\end{proposition}

\subsection{CCR and type I groupoids}
\label{sec:ccr-type-I}

We refer the reader to \cite{murphy90} for the basic theory of \Cstar-algebras.  To each of the objects considered in Sections \ref{sec:inverse-semigroups} and \ref{sec:ample-groupoids} one can associate a \Cstar-algebra.  For groupoid \Cstar-algebras $\Cstar(\G)$, we refer the reader to \cite{renault80}.  The \Cstar-algebras $\Cstar(S)$ associated with inverse semigroups are explained in \cite{kumjian1984, paterson1999}.  In particular, it is known that $\Cstar(S) \cong \Cstar(\G(S))$ canonically.  We refer to \cite[Section 5.6]{murphy90} for details on the following two notions from representation theory of \Cstar-algebras.

\begin{definition}
  \th\label{def:ccr-type-I}
  Let $A$ be a \Cstar-algebra.  Then $A$ is called \emph{CCR} if the image of every irreducible *-representation of $A$ equals the compact operators.  We call $A$ \emph{GCR} or \emph{type I} if the image of every irreducible *-representation of $A$ contains the compact operators.

  Inverse semigroups and a groupoids are called CCR or type I, respectively, if their \Cstar-algebras have this property.
\end{definition}

In this article, the notions of CCR and type I are accessed solely through the following special case of a result of Clark combined with the characterisation of discrete type I groups by Thoma.
\begin{theorem}[{\cite[Theorems 1.3 and 1.4]{clark07} and \cite{thoma68}}]
  \th\label{thm:clark}
  Let $\G$ be a second countable, {\'e}tale, Hausdorff groupoid.   Then $\G$ is CCR if and only if all the isotropy groups are virtually abelian and the orbit space of $\G$ is \Tone.   Further,  $\G$ is type I if and only if all the isotropy groups are virtually abelian and the orbit space of $\G$ is \Tzero.  
\end{theorem}

\section{Separation properties of orbit spaces}
\label{sec:separation-properties}

In this section we consider separation properties of orbit spaces and provide algebraic characterisations of ample groupoids whose orbit space is a \Tone-space and a \Tzero-space, respectively.

Let us start by introducing the simplest example of a groupoid whose orbit space is not \Tone.
\begin{example}
  \th\label{ex:btone}
  Let $\G_\Tone$ be the groupoid associated with the equivalence relation $\{(n,m) \in (\N \cup \{\infty\})^2 \mid n + m < \infty\}$.  Denote by $B_\Tone = \Gamma(\G_\Tone)$ the Boolean inverse semigroup of compact open bisections of $\G_\Tone$.  Then $B_\Tone$ admits a presentation with generators $(s_n)_{n \in \N}$ and $f$ satisfying the relations
  \begin{align*}
    & \im s_n = \supp s_{n+1} & \text{ for all } n \\
    & \supp s_n \perp \supp s_{n + 1} & \text{ for all } n \neq m \\
    & f^2  = f \geq \supp s_n & \text{ for all } n
  \end{align*}
\end{example}

\begin{proposition}
  \th\label{prop:tone}
Let $\G$ be an ample Hausdorff groupoid. Then the orbit space of $\G$ is not $\Tone$ if and only if $\Gamma(\G)$ has $B_{\Tone}$ as a subquotient.
\end{proposition}
\begin{proof}
  Assume that the orbit space of $\G$ is not $\Tone$.  Then by \th\ref{prop:tzero-characterisations} there is some non-closed orbit, say $\G x \subseteq \Gnought$.  Let $(x_n)_{n \in \N}$ be a convergent sequence from $\G x$ whose limit $x_\infty = \lim x_n$ does not lie in $\G x$.  Since $\G$ is ample, there are bisections $(s_n)_{n \in \N}$ in $\Gamma(\G)$ such that $s_n \cap d^{-1}(x_n)  \cap r^{-1}(x_{n+1}) \neq \emptyset$ for all $n \in \N$.  Without loss of generality, we may assume that $(\supp s_n)_n$ are pairwise disjoint subsets of $\Gnought$.  Let $B \subseteq \Gamma(\G)$ be the Boolean inverse semigroup generated by all $(s_n)_{n \in \N}$ together with the idempotents of $\Gamma(\G)$.  Denote by $\cH = \bigcup B$ the open subgroupoid of $\G$ associated with $B$.  The subset $A = \{x_n \mid n \in \N \cup \{\infty\}\} \subseteq \Gnought$ is $\cH$-invariant, so that by \th\ref{prop:restriction-induces-bis-map} there is a quotient map of Boolean inverse semigroups $B \cong \Gamma(\cH) \to \Gamma(\cH|_A)$.  If suffices to note that $\cH|_A \cong \cG_{\Tone}$ so that $\Gamma(\cH|_A) \cong B_{\Tone}$ follows.

  Assume now that $B_\Tone$ is a subquotient of $\Gamma(\G)$.  Denote by $(s_n)_{n \in \N}$ and $f$ preimages in $\Gamma(\G)$ of the generators of $B_\Tone$.  Replacing $s_n$ by $s_n f$, we may suppose that $\supp s_n \leq f$ holds for all $n \in \N$.  Further, writing $p_n = \bigvee_{k < n} \supp s_k$ and replacing $s_n$ by $s_n(f - p_n)$, we may assume that $(\supp s_n)_n$ are pairwise orthogonal.  Write $t_n = s_{n-1} \dotsm s_0$ and $q_n = t_n^* (\supp s_n) t_n$.  Then every finite subfamily of $(q_n)_{n \in \N}$ has a non-zero meet, since the same statement holds true for their images in $B_\Tone$.  Denoting by $U_n \subset \Gnought$ the compact open subset corresponding to $q_n$, it follows that $\bigcap_{n \in \N} U_n \neq \emptyset$.  Choose $x_0$ in there and define $x_n = t_n x_0 \in \supp s_n$.  Since $\supp s_n \leq f$ for all $n \in \N$, there is a convergent subsequence $(x_{n_k})_k$ of $(x_n)$.  So the orbit $\G x_0$ has an accumulation point, which proves that the orbit space of $\G$ is not $\Tone$.
\end{proof}

\begin{remark}
  \th\label{rem:tone-no-corner}
  Comparing the statement of \th\ref{prop:tone} with \th\ref{prop:tzero}, it is natural to ask whether it is possible to find a corner of $\Gamma(\G)$ that has $B_\Tone$ as a quotient, rather than finding $B_\Tone$ as a subquotient of $\Gamma(\G)$.  This is not possible as the following example show.  We consider the groupoid $\G$ arising from the equivalence relation on $\beta \N$, the Stone-{\u C}ech compactification of $\N$, that relates all elements from $\N$, and nothing else. It is straight forward to check that $\G$ is an ample groupoid, since every point of $\N$ is isolated in $\beta \N$.  Corners of $\Gamma(\G)$ correspond to restrictions of $\G$ to compact open subsets, as explained in Section \ref{sec:nc-stone-duality}.  Compact open subsets of $\beta \N$ are exactly of the form $U = \overline{D}$ for subsets $D \subset \N$.  Clearly if $D$ is finite, $\G_\Tone$ cannot arise as a restriction of $\G|_U = \G|_D$.  But every infinite subset of $\N$ has infinitely many accumulation points in $\beta \N$, so $\G_\Tone$ cannot arise as a restriction of $\G$ at all.
\end{remark}

\begin{proposition}
  \th\label{prop:tzero}
  Let $\G$ be a second countable, ample Hausdorff groupoid. Then the following statements are equivalent.
  \begin{itemize}
  \item The orbit space of $\G$ is not $\Tzero$.
  \item A corner of $\Gamma(\G)$ has an infinite, monoidal and $0$-simplifying quotient.
  \item There is an infinite, monoidal and $0$-simplifying subquotient of $\Gamma(\G)$.
  \end{itemize}
\end{proposition}
\begin{proof}
  Assume first that the orbit space of $\G$ is not $\Tzero$.  By \th\ref{prop:tzero-characterisations}, there exists a self-accumulating orbit of $\G$.  Denote its closure by $A$.  Then $A$ is a $\G$-invariant subset of $\Gnought$, so that by \th\ref{prop:restriction-induces-bis-map} the restriction to $A$ induces a quotient map $\Gamma(\G) \to \Gamma(\G|_A)$.  The groupoid $\G|_A$ has a dense orbit, so that by \cite[Lemma 3.4]{steinberg2019} the set of its units with a dense orbit is comeager.  In particular, there is a compact open subset $U \subseteq A$ such that every point of $U$ has a dense $\G|_A$-orbit.  Note that since $U \subseteq A$ is open, $\G|_U$ is {\'e}tale.  Further, since $U$ is compact open, $\Gamma(\G|_U)$ is a corner of $\Gamma(\G|_A)$.  We note that $A$ is infinite, since it is self-accumulating, so that also $U$ is infinite.  Hence $\Gamma(\G|_U)$ is infinite.  Further, $\Gamma(\G|_U)$ is monoidal, since $U$ is compact.  It is $0$-simplifying by \th\ref{prop:0simplifying}, since $\G|_U$ is minimal.  If $V \subseteq \Gnought$ denotes any compact open subset such that $V \cap A = U$, then the quotient map $\Gamma(\G) \to \Gamma(\G|_A)$ maps $\Gamma(\G|_V)$ onto $\Gamma(\G|_U)$.  So $\Gamma(\G|_U)$ is a quotient of a corner of $\Gamma(\G)$.

  If $\Gamma(\G)$ has a corner with an infinite, monoidal and $0$-simplifying quotient, then it is a subquotient of $\Gamma(\G)$.  So let us assume that there is an infinite, monoidal and $0$-simplifying subquotient of $\Gamma(\G)$. We will show that the orbit space of $\G$ is not $\Tzero$.  Write $\Gamma(\G) \supset C \thra B$ for the given subquotient.  Choosing an idempotent preimage $p \in \rE(C)$ of the unit of $B$, we obtain a surjection $p C p \thra B$.  Let $V \subset \Gnought$ be the compact open subset corresponding to $p$.  Replacing $C$ by $pCp$, we thus find a unital inclusion $\Gamma(\G|_V) \supset C$ and a quotient map $\pi: C \thra B$.  Let $\cH$ be the ample groupoid associated with $C$ by noncommutative Stone duality, that is $\Gamma(\cH) \cong C$.  Write $X = \cH^{(0)}$.  Considering the restriction of $\pi$ to idempotents, we find a closed subset $A \subseteq X$ such that the following diagram commutes.
  \begin{gather*}
    \xymatrix{
      \CO(X) \ar[r]^{\pi} \ar[d]_{\res_A} & \rE(B) \\
      \CO(A) \ar@{-->}[ru]_{\cong}
      }
    \end{gather*}
    Let $x \in A$.  We will show that $\cH x \cap A$ is self-accumulating.  By noncommutative Stone duality, there is an ample groupoid $\cK$ such that $\Gamma(\cK) \cong B$.  From the fact that $B$ is infinite, monoidal and 0-simplifying, it follows that $\cK$ has is infinite, has a compact unit space and is minimal.  In particular, the orbits of $\cK$ are self-accumulating.  So for every compact open neighbourhood $U \subseteq A$ of $x$, there is some bisection $t \in \Gamma(\cK)$ such that $x \in \supp t$ and $x \notin \im t \leq U$.  Let $s \in \pi^{-1}(t)$ denote a preimage of $t$.  Then $\supp s \cap A = \pi(\supp s) = \supp \pi(s) = \supp t$ and similarly $\im s \cap A = \im t$.  It follows that $\cH x \cap A$ and thus also the $\cH$-orbit of $x$ is self-accumulating.  Thus the orbit space of $\cH$ is not \Tzero.  Consider now the surjective map $\varphi: V \to X$ which is dual to the inclusion $\CO(X) \subset \CO(V)$.  Given $x,y \in X$ in the same $\cH$-orbit, there is $s \in C$ such that $sx = y$.  Let $u \in V$ be some preimage of $x$ under $\varphi$.  Considering $s$ as a bisection of $\G|_V$, we define $v = su$, which lies in the same $\G|_V$-orbit as $u$ and satisfies $\varphi(v) = y$.  This shows that every $\cH$-orbit is contained in the image of a $\G|_V$-orbit.  In particular, there is some orbit of $\G|_V$ that is not finite and hence not locally closed.  So \th\ref{prop:tzero-characterisations} says that the orbit space of $\G$ is not \Tzero.    
\end{proof}

\section{Group quotients and isotropy groups}
\label{sec:isotropy-groups}

In this section, we relate the isotropy groups of an ample groupoid with certain subquotients of the Boolean inverse semigroup of its compact open bisections.  This result will allow us to address the condition on isotropy groups from \cite{clark07}.

\begin{proposition}
  \th\label{prop:isotropy-groups-characterisation}
  Let $\G$ be a second countable, ample Hausdorff groupoid whose orbit space is \Tzero.  Let $x \in \Gnought$ be a unit, write $G =\G|_x$ for the isotropy group at $x$ and denote by $G_0$ the associated group with zero.  Then $G_0$ is a quotient of a corner of $\Gamma(\G)$.  Vice verse, if $\G$ is any ample Hausdorff groupoid and $G$ is a group such that $G_0$ is a quotient of a corner of $\G$, then $G$ is a quotient of a point stabiliser of $\G$.
\end{proposition}
\begin{proof}
  Since the orbit space of $\G$ is assumed to be {\Tzero} and $\G$ is second countable, its orbits are locally closed by \th\ref{prop:tzero-characterisations}.  Let $U \subseteq \Gnought$ be a compact open neighbourhood of $x$, such that $\G x \cap U$ is closed in $U$ and hence compact.  Since $\G$ is {\'e}tale, we know that $\G x$ is countable, so that $\G x \cap U$ is actually finite.  We may thus shrink $U$ so that $\G x \cap U = \{x\}$ holds.  Denote by $p = U \in \Gamma(\G)$ the idempotent bisection associated with $U$. Then the  corner of $\Gamma(\G)$ can be identified as $p \Gamma(\G) p = \Gamma(\G|_U)$.  Since $x$ is fixed by $\G|_U$, the restriction from $U$ to $\{x\}$ induces a quotient of Boolean inverse semigroups $\Gamma(\G|_U) \to \Gamma(\G|_x) = (\G|_x)_0$ by \th\ref{prop:restriction-induces-bis-map}.

  Let us know assume that $\G$ is any ample Hausdorff groupoid, let $G$ be a group and $p \in \Gamma(\G)$ an idempotent, for which there is a quotient map $\pi: p\Gamma(\G)p \thra G_0$.  Denote by $U \subseteq \Gnought$ the compact open subset corresponding to $p$.  Then there is a natural isomorphism $p \Gamma(\G) p \cong \Gamma(\G|_U)$.  Since the algebra of idempotents of $G_0$ is trivial, $\pi|_{\rE(\Gamma(\G|_U))}$ is a character.  By Stone duality, there is $x \in U$ such that $\pi|_{\rE(\Gamma(\G|_U))} = \ev_x$.  In particular, $\{x\} \subseteq U$ is a $\G|_U$-invariant subset by Lemma \ref{lem:restriction-implies-invariant}.  So by the universal property of the restriction map described in \th\ref{prop:restriction-induces-bis-map}, the homomorphism $\pi$ factors through $\Gamma(\G|_U) \to \Gamma(\G|_x) = (\G_x)_0$.  So $G_0$ is a quotient of $(\G|_x)_0$, which implies that $G$ is a quotient of $\G|_x$.
\end{proof}

\begin{example}
  \th\label{ex:subquotient-groups-not-sufficient}
  It might be tempting to admit arbitrary subquotients of $\Gamma(\G)$ in the statement of Proposition \ref{prop:isotropy-groups-characterisation}, however this does not even suffice under the condition that the orbit space of $\G$ is \Tone.  Indeed, the topological full group of $\G$ is always a subgroup of $\Gamma(\G)$, and it can be large even if $\G$ is effective.  For example the topological full group associated with $B_\Tone$ is $\Sym(\N)$.
\end{example}

We next formulate an appropriate version of Proposition \ref{prop:isotropy-groups-characterisation} for inverse semigroups.  Let us start with a short lemma relating quotients of an inverse semigroup and its booleanization.
\begin{lemma}
  \th\label{lem:quotient-semigroup-Boolean-envelop}
  Let $S$ be an inverse semigroup and $B(S) \thra G_0$ a quotient of its booleanization.  Then $S_0 \to G_0$ is surjective.
\end{lemma}
\begin{proof}
  Denote the quotient map $B(S) \thra G_0$ by $\pi$ and let $g \in G$.  Using the description of $B(S)$ presented in Section \ref{sec:inverse-semigroups}, there is some preimage $\sum_i t_i e_i \in B(S)$ of $g$. Since $G_0$ has only two idempotents, $\pi|_{E(B(S))}$ is a character.  So there is a unique $i_0$ satisfying $\pi(e_{i_0}) = 1$.  This implies $\pi(\sum_i t_i e_i) = \pi(t_{i_0})$, showing that $g \in \pi(S_0)$.
\end{proof}

\begin{proposition}
  \th\label{prop:isotropy-groups-characterisation-semigroups}
  Let $S$ be a countable inverse semigroup such that the orbit space of $\G = \G(S)$ is \Tzero.  Let $x \in \Gnought$ and write $G = \G\vert_x$.  Then the group with zero $G_0$ is a quotient of a corner of $S_0$.
\end{proposition}
\begin{proof}
  Fix $x \in \Gnought$.  Since $\Gnought \cong \widehat{E(S)}$, there is $q \in E(S)$ such that $x(q) = 1$.  Denote by $U \subseteq \Gnought$ the compact open subset corresponding to $q$.  As in the proof of Proposition \ref{prop:isotropy-groups-characterisation}, the fact that orbits of $\G$ are locally closed implies that $\G x \cap U$ is finite.  Fix an enumeration $x = x_0, x_1, \dotsc, x_n$ of $\G x \cap U$.  For $i \in \{1, \dotsc, n\}$ there is $q_i \in E(S)$ such that $x_0(q_0) = 1$ and $x_i(q_i) = 0$.  Put $p = q \cdot q_1 \dotsm q_n$ and let $V \subseteq \Gnought$ be the compact open subset corresponding to $p$.  Then the identification of corners of Boolean inverse semigroups says that
  \begin{gather*}
    \G(pSp) = \G(B(pSp)) \cong \G(pB(S)p) \cong \G|_V.
  \end{gather*}
  Since $x \in V$ is $\G|_V$-fixed, there is a quotient map $\Gamma(\G|_V) \to \Gamma(\G|_x) \cong (\G|_x)_0$.  The identification $\Gamma(\G|_V) \cong B(pSp)$ shows that there is a surjection $B(pSp) \to (\G|_x)_0$.  By \th~\ref{lem:quotient-semigroup-Boolean-envelop} its restriction $p S_0 p \cong (pSp)_0 \to (\G|_x)_0$ remains surjective.
\end{proof}

\section{Proof of the main results}
\label{sec:main-proofs}

We now obtain the proof of our main theorems. Thanks to the preparation made in the previous sections, all proofs are rather similar and we spell out details only for \th\ref{introthm:ccr-groupoid}.
\begin{proof}[Proof of \th\ref{introthm:ccr-groupoid}]
  Since $\G$ is an {\'e}tale, second countable Hausdorff groupoid, we may appeal to the results of Clark and Thoma described in \th\ref{thm:clark}.  It follows that $\G$ is CCR if and only if all its isotropy groups are virtually abelian and its orbit space is \Tone.  We can thus combine \th\ref{prop:isotropy-groups-characterisation} and \th\ref{prop:tone} to complete our proof.
\end{proof}

\begin{proof}[Proof of \th\ref{introthm:type-I-groupoid}]
Upon replacing the reference to \th\ref{prop:tone} by \th\ref{prop:tzero}, the same argument as used in the proof of \th\ref{introthm:ccr-groupoid} can be applied.
\end{proof}

\begin{proof}[Proof of \th\ref{introthm:ccr-semigroup}]
  Replacing the reference to \th\ref{prop:isotropy-groups-characterisation} by \th\ref{prop:isotropy-groups-characterisation-semigroups}, and alluding to the fact that $\Gamma(\G(S)) \cong B(S)$, the proof of \th\ref{introthm:ccr-groupoid} applies. 
\end{proof}

\begin{proof}[Proof of \th\ref{introthm:type-I-semigroup}]
  Making the same adaptions as in the passage from the proof of \th\ref{introthm:ccr-groupoid} to \ref{introthm:type-I-groupoid}, the proof of \th\ref{introthm:ccr-semigroup} applies.
\end{proof}

\begin{remark}
  \label{rem:no-quotient-semigroup}
  In view of our construction of $B(S)$ described in Section \ref{sec:inverse-semigroups}, the booleanization of an inverse semigroup can be concretely calculated, so that the conditions on $B(S)$ in Theorems \ref{introthm:ccr-semigroup} and \ref{introthm:type-I-semigroup} can checked.  Specifically for Theorem \ref{introthm:type-I-semigroup} we remark that the groupoid $\G(S)$ associated with an inverse semigroup always has a fixed point (whose isotropy group is the maximal group quotient of $S$).  It corresponds to the trivial character on $E(S)$ that maps every idempotent to $1$.  Since $0$-simplifying Boolean inverse semigroups correspond to minimal groupoids, it is not possible to directly translate the condition on $B(S)$ in Theorem \ref{introthm:type-I-semigroup} to an algebraic statement about $S$ itself.
\end{remark}

\printbibliography

\vspace{2em}
\begin{minipage}[t]{0.45\linewidth}
  \small
  Gabriel Favre \\
  Department of Mathematics \\
  Stockholm University \\
  106 91 Stockholm \\
  Sweden \\[1em]
  favre@math.su.se
\end{minipage}
\begin{minipage}[t]{0.45\linewidth}
  \small
  Sven Raum \\
  Department of Mathematics \\
  Stockholm University \\
  106 91 Stockholm \\
  Sweden \\[1em]
  and \\[1em]
  Institute of Mathematics of the \\ Polish Academy of Sciences \\
  ul.\ \'Sniadeckich 8 \\
  00-656 Warszawa \\
  Poland \\[1em]
  raum@math.su.se
\end{minipage}

\end{document}